\documentclass[a4paper,11pt]{amsart}

\usepackage{amssymb}
\usepackage[pagebackref,
    ,pdfborder={0 0 0}
    ,urlcolor=black,a4paper,hypertexnames=false]{hyperref}
\hypersetup{pdfauthor={Clara L\"oh, Roman Sauer},pdftitle=Bounded cohomology of amenable covers}

\usepackage{pgf}
\usepackage{tikz}
\usetikzlibrary{decorations.pathmorphing,calc}
\tikzstyle{every picture}+=[remember picture]

\usepackage[all]{xy}
\SelectTips{cm}{}

\newtheorem{thm}{Theorem}[section]
\newtheorem{prop}[thm]{Proposition}
\newtheorem{lem}[thm]{Lemma}
\newtheorem{cor}[thm]{Corollary}

\newtheorem{setup}[thm]{Setup}

\theoremstyle{remark}
\newtheorem{rem}[thm]{Remark}
\newtheorem{example}[thm]{Example}

\theoremstyle{definition}
\newtheorem{defi}[thm]{Definition}

\usepackage{amsfonts}
\newcommand{\Z}{\mathbb{Z}}

\newcommand{\R}{\mathbb{R}}
\newcommand{\N}{\mathbb{N}}

\DeclareMathOperator{\cone}{cone}
\DeclareMathOperator{\id}{id}
\DeclareMathOperator{\Hom}{Hom}

\DeclareMathOperator{\Amen}{{\sf Am}}
\DeclareMathOperator{\Mcyc}{{\sf MCyc}}

\DeclareMathOperator{\im}{im}

\def\fa#1{%
  \forall_{#1}\;\;\;}

\def\args{\;\cdot\;}

\def\longrightarrow{\rightarrow}
\def\longmapsto{\mapsto}

\def\ucov#1{%
  \widetilde{#1}
}

\DeclareMathOperator{\Eop}{E}
\def\EFG#1#2{%
  \Eop_{#1}{#2}}

\def\hbadmissible{%
  $H_b^*$-admissible}
\def\ladmissible{%
  $H^{\ell^1}_*$-admissible}

\usepackage{color}
\usepackage{pdfcolmk}

\def\draftinfo{}

\author[C.~L\"oh]{Clara L\"oh}
\address{Fakult\"at f\"ur Mathematik\\
         Universit\"at Regensburg\\
         93040~Regensburg\\
         }
\email{clara.loeh@mathematik.uni-r.de}

\author[R.~Sauer]{Roman Sauer}

\address{Karlsruhe Institute of Technology\\
76131~Karlsruhe\\}
\email{roman.sauer@kit.edu}

\title[Bounded cohomology of amenable covers]
      {Bounded cohomology of amenable covers\\ via classifying spaces}
\date{\today.\ \copyright{\ C.~L\"oh, R.~Sauer 2019}. 
    This work was supported by the CRC~1085 \emph{Higher Invariants} 
    (Universit\"at Regensburg, funded by the DFG) and by the RTG 2229
    \emph{Asymptotic Invariants and Limits of Groups and Spaces}
    (KIT, funded by the~DFG). 
    \draftinfo\\
     MSC~2010 classification: 55N10, 55N35} 
\begin{document}

\begin{abstract}
  Gromov and Ivanov established an analogue of Leray's theorem
  on cohomology of contractible covers for bounded cohomology
  of amenable covers. We present an alternative proof of this
  fact, using classifying spaces of families of subgroups.
\end{abstract}

\maketitle

\section{Introduction}

The idea that the cohomology of a space can be computed as
cohomology of the nerve of an open cover consisting of contractible
subsets first appeared in a paper by Weil~\cite{weil}, which was
preceded by a paper of Leray~\cite{leray} with a similar idea. 

Gromov and Ivanov established partial analogues
for bounded cohomology in terms of covers that consist of amenable
subsets; we will consider the following version of this phenomenon:

\begin{thm}\label{thm:main}
  Let $X$ be a path-connected CW-complex, let
  $U$ be an amenable open cover of~$X$,
  let $N$ be the nerve of~$U$, and let $|N|$ be the geometric 
  realisation of~$N$. Let~$c_X \colon H^*_b(X;\R) \longrightarrow H^*(X;\R)$ be 
  the comparison map from bounded to ordinary cohomology. Then the following hold. 
  \begin{enumerate}
  \item If $U$ is convex, then $c_X$ factors through the nerve map~$\nu \colon X \longrightarrow |N|$: 
  More precisely, there is an $\R$-linear map~$\varphi \colon H^*_b(X;\R)
  \longrightarrow H^*(|N|;\R)$ with
  \[ H^*(\nu;\R) \circ \varphi = c_X.
  \]
  \begin{align*}
    \xymatrix{%
    H^*_b(X;\R) \ar[r]^-{c_X} \ar@{-->}[dr]_-{\varphi}
    & H^*(X;\R)
    \\
    & H^*(|N|;\R) \ar[u]_-{H^*(\nu;\R)}}
  \end{align*}
  \item If the multiplicity of $U$ is at most $m$, then the comparison
    map $c_X$ vanishes in all degrees $\ast\ge m$.
  \end{enumerate}
\end{thm}

Here, an \emph{amenable cover} of~$X$ is an open cover~$U$ of~$X$ by
path-connected sets with the following property: For each~$V$ in~$U$
and each~$x \in V$, the image of the homomorphism~$\pi_1(V,x)
\longrightarrow \pi_1(X,x)$ induced by the inclusion~$V
\hookrightarrow X$ is an amenable subgroup of~$\pi_1(X,x)$
(Definition~\ref{def:amcover}). Such an amenable cover~$U$ is \emph{convex}
if all finite intersections of members of~$U$ are path-connected or empty
(Definition~\ref{def:convex}).

The first statement of Theorem~\ref{thm:main} was proved by
Ivanov~\cite[Section~6]{ivanov} using sheaf cohomology and a spectral
sequence computation. The assumption on convexity is missing in
Ivanov's paper but needed~(see Example~\ref{exa:ex2}).  The second
statement of Theorem~\ref{thm:main} is Gromov's vanishing
theorem~\cite[Section~3]{vbc} whose proof is based on the theory of
multicomplexes. Recently, Frigerio and Moraschini reworked Gromov's
theory of multicomplexes and gave a proof of
Theorem~\ref{thm:main}~\cite[Section~6]{frigeriomoraschini}.

In this note, we present an alternative proof of Theorem~\ref{thm:main}
that only uses standard facts about bounded cohomology (via strong
relatively injective resolutions) and the classifying space
for the family of amenable subgroups. More precisely, we separate
the proof into a statement about admissibility of the family of
amenable subgroups (Section~\ref{subsec:admi}) and generic
properties of classifying spaces of families (Section~\ref{subsec:admiHb}).
The abstract version of Theorem~\ref{thm:main} is Theorem~\ref{thm:admissible}.

Furthermore, our approach also leads to a straightforward proof
for the corresponding statement in $\ell^1$-homology (and of the
corresponding abstract version: Theorem~\ref{thm:admissiblel1}):

\begin{thm}\label{thm:mainl1}
  Let $X$ be a path-connected CW-complex, let
  $U$ be an amenable open cover of~$X$,
  let $N$ be the nerve of~$U$, and let $|N|$ be the geometric
  realisation of~$N$. Let~$c^{\ell^1}_X \colon H_*(X;\R) \longrightarrow H^{\ell^1}_*(X;\R)$ 
  be the comparison map from ordinary to $\ell^1$-homology. Then the following hold. 
  \begin{enumerate}
  \item If $U$ is convex, then~$c^{\ell^1}_X$   
  factors through the nerve map~$\nu \colon X \longrightarrow |N|$: \\
  More precisely, there is an $\R$-linear map~$\varphi \colon H_*(|N|;\R)
  \longrightarrow H^{\ell^1}_*(X;\R)$ with
  \[ \varphi \circ H_*(\nu;\R) = c^{\ell^1}_X.
  \]
  \begin{align*}
    \xymatrix{%
      H_*(X;\R) \ar[r]^-{c^{\ell^1}_X} \ar[d]_-{H_*(\nu;\R)}
    & H^{\ell^1}_*(X;\R)
    \\
    H_*(|N|;\R) \ar@{-->}[ur]_-{\varphi}
    & }
  \end{align*}
  \item If the multiplicity of $U$ is at most~$m$, then the comparison map~$c^{\ell^1}_X$ vanishes in all degrees $\ast\ge m$. 
   \end{enumerate}
\end{thm}

In particular, as a special case, we obtain the corresponding
vanishing theorem for $\ell^1$-homology and $\ell^1$-invisibility
results for amenable convex open covers on CW-complexes
(Corollary~\ref{cor:vanishing}), established recently by
Frigerio~\cite{frigeriol1}. Using Lemma~4.1. in \emph{loc. cit.},
Frigerio reduces the statement for topological spaces to the one for
CW-complexes and thus can drop the assumption on the space being a
CW-complex. Using the same lemma we may also drop the assumption on
$X$ being a CW-complex in Theorems~\ref{thm:main}~(2)
and~Theorem~\ref{thm:mainl1}~(2).

\subsection{Two non-examples}

We give two simple examples complementing the hypotheses and the
conclusion in Theorem~\ref{thm:main}.

Both examples are based on the oriented closed connected surface~$\Sigma$
of genus~$2$. Because $\Sigma$ admits a hyperbolic structure, the comparison
map
\[ H_b^2(\Sigma;\R) \longrightarrow H^2(\Sigma;\R) \cong \R
\]
is surjective (and thus non-trivial)~\cite[p.~9/17]{vbc}. 

\def\eight{%
  \draw[->] (0:1) -- (45:1);
  \draw[->] (45:1) -- (90:1);
  \draw[->] (135:1) -- (90:1);
  \draw[->] (180:1) -- (135:1);
  \draw[->] (180:1) -- (225:1);
  \draw[->] (225:1) -- (270:1);
  \draw[->] (315:1) -- (270:1);
  \draw[->] (0:1) -- (315:1);
  \draw (22.5:1.2) node {$a_1$};
  \draw (67.5:1.2) node {$b_1$};
  \draw (112.5:1.2) node {$a_1$};
  \draw (157.2:1.2) node {$b_1$};
  \draw (202.5:1.2) node {$a_2$};
  \draw (247.5:1.2) node {$b_2$};
  \draw (292.5:1.2) node {$a_2$};
  \draw (347.5:1.2) node {$b_2$};
}
\def\eightpath{%
  (0:1) -- (45:1) -- (90:1) -- (135:1) -- (180:1) -- (225:1) -- (270:1) -- (315:1) -- cycle;
}
\def\blau{blue!70!green}

\begin{example}\label{exa:ex1}
  In general, in Theorem~\ref{thm:main}, there is \emph{no}
  such factorisation~$\varphi$ that is bounded:  
  We consider the open cover~$U$ of~$\Sigma$ depicted in
  Figure~\ref{fig:ex1}. All members of~$U$ are contractible
  or homotopy equivalent to~$S^1$. Therefore, $U$ is an amenable
  cover; moreover, one easily checks that the cover is convex
  (in the sense of Definition~\ref{def:convex}).

  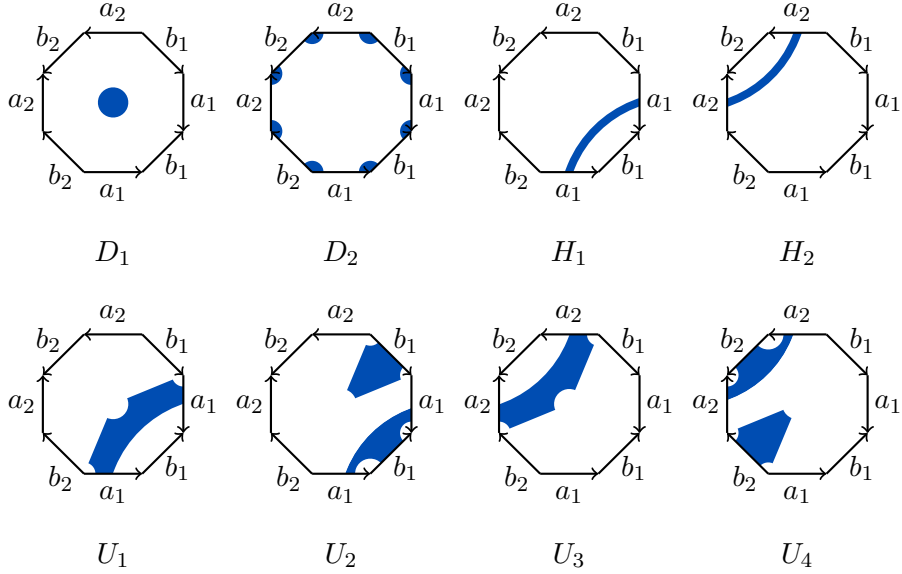
\begin{figure}
    \begin{center}
      \begin{tikzpicture}[x=1cm,y=1cm,thick,rotate={-112.5}]
        \begin{scope}[shift={(0,0)}]
          \eight
          \fill[\blau] (0,0) circle (0.2);
          \draw (22.5:2) node {$D_1$};
        \end{scope}  
        \begin{scope}[shift={(112.5:3)}]
          \begin{scope}
            \clip\eightpath;
            \foreach \j in {0,...,7} {%
              \fill[\blau] (\j*45:1) circle (0.15);}
          \end{scope}
          \eight
          \draw (22.5:2) node {$D_2$};
        \end{scope}  
        \begin{scope}[shift={(112.5:6)}]
          \begin{scope}
            \clip\eightpath;
            \draw[line width=3,\blau] (67.5:2) circle (1.5);
          \end{scope}
          \eight
          \draw (22.5:2) node {$H_1$};
        \end{scope}  
        \begin{scope}[shift={(112.5:9)}]
          \begin{scope}
            \clip\eightpath;
            \begin{scope}[rotate={180}]
            \draw[line width=3,\blau] (67.5:2) circle (1.5);
            \end{scope}
          \end{scope}
          \eight
          \draw (22.5:2) node {$H_2$};
        \end{scope}  

        \begin{scope}[shift={(22.5:4)}]
          \begin{scope}
            \clip\eightpath;
            \fill[\blau]\eightpath;
            \fill[line width=3,white] (67.5:2) circle (1.5);
            \fill[white] (0,0) circle (0.2);  
            \fill[white] (45:1) circle (0.2); 
            \foreach \j in {0,...,7} {%
              \fill[white] (\j*45:1) circle (0.15);}
            \fill[white] (0,0) -- (135:1) arc (135:360:1);
          \end{scope}
          \eight
          \draw (22.5:2) node {$U_1$};
        \end{scope}  
        \begin{scope}[shift={($(22.5:4)+(112.5:3)$)}]
          \begin{scope}
            \clip\eightpath;
            \fill[\blau] (0,0) -- (135:1) -- (180:1) -- cycle;
            \fill[\blau] (67.5:2) circle (1.45);
            \draw[line width=3,white] (67.5:2) circle (1.5);
            \fill[white] (0,0) circle (0.2);  
            \fill[white] (45:1) circle (0.2); 
            \foreach \j in {0,...,7} {%
              \fill[white] (\j*45:1) circle (0.15);}
          \end{scope}
          \eight
          \draw (22.5:2) node {$U_2$};
        \end{scope}  
        \begin{scope}[shift={($(22.5:4) + (112.5:6)$)}]
          \begin{scope}[rotate={180}]
            \clip\eightpath;
            \fill[\blau]\eightpath;
            \fill[line width=3,white] (67.5:2) circle (1.5);
            \fill[white] (0,0) circle (0.2);  
            \fill[white] (45:1) circle (0.2); 
            \foreach \j in {0,...,7} {%
              \fill[white] (\j*45:1) circle (0.15);}
            \fill[white] (0,0) -- (135:1) arc (135:360:1);
          \end{scope}
          \eight
          \draw (22.5:2) node {$U_3$};
        \end{scope}  
        \begin{scope}[shift={($(22.5:4)+(112.5:9)$)}]
          \begin{scope}[rotate={180}]
            \clip\eightpath;
            \fill[\blau] (0,0) -- (135:1) -- (180:1) -- cycle;
            \fill[\blau] (67.5:2) circle (1.45);
            \draw[line width=3,white] (67.5:2) circle (1.5);
            \fill[white] (0,0) circle (0.2);  
            \fill[white] (45:1) circle (0.2); 
            \foreach \j in {0,...,7} {%
              \fill[white] (\j*45:1) circle (0.15);}
          \end{scope}
          \eight
          \draw (22.5:2) node {$U_4$};
        \end{scope}  
      \end{tikzpicture}
    \end{center}

    \caption{The open cover of~$\Sigma$ in Example~\ref{exa:ex1}}
    \label{fig:ex1}
  \end{figure}
  
  Then the nerve~$N$ of~$U$ satisfies~$|N| \simeq S^1 \lor S^2 \lor S^1$
  (the $S^2$ stems from the octahedron spanned by the sets~$D_1$, $D_2$,
  $U_1$, \dots, $U_4$). 

  We will now show that all non-trivial classes in~$H^2(|N|;\R)$ are
  unbounded: To this end, we consider the inclusion~$i \colon S^2
  \longrightarrow S^2 \lor S^2 \lor S^1 \simeq |N|$. Then the induced map~$H^2(i;\R)
  \colon H^2(|N|;\R) \longrightarrow H^2(S^2;\R)$ is an
  isomorphism. However, all non-trivial classes in~$H^2(S^2;\R)$ are
  known to be unbounded~\cite[p.~8/17]{vbc}. Therefore, also all
  non-trivial classes in~$H^2(|N|;\R)$ are unbounded.

  In particular, the surjection~$H_b^2(\Sigma;\R) \longrightarrow
  H^2(\Sigma;\R) \cong \R$ does \emph{not} admit a 
  factorisation~$H_b^2(\Sigma;\R) \longrightarrow H^2(|N|;\R)$ over
  the nerve map by a bounded linear map.
\end{example}

\begin{example}\label{exa:ex2}
  In general, Theorem~\ref{thm:main} does \emph{not} hold without
  the convexity condition: 
  We consider the following open cover~$U$ of~$\Sigma$ depicted in
  Figure~\ref{fig:ex2}. 

  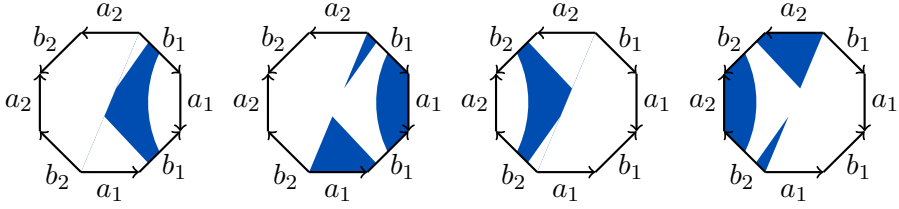
\begin{figure}
    \begin{center}
      \begin{tikzpicture}[x=1cm,y=1cm,thick,rotate={-112.5}]
        \begin{scope}[shift={(0,0)}]
          \begin{scope}
            \clip\eightpath;
            \fill[\blau] (1,0) arc (0:180:1);
            \fill[white] (112.5:2) circle (1.5);
            \fill[white] (1,0) arc (0:55:1) -- (0.2,0) -- cycle;
            \fill[white] (170:1) arc (170:180:1) -- (-0.2,0) -- cycle;
          \end{scope}
          \eight
        \end{scope}  
        \begin{scope}[shift={(112.5:3)}]
          \begin{scope}
            \clip\eightpath;
            \fill[\blau] (112.5:2) circle (1.5);
            \fill[\blau] (1,0) arc (0:55:1) -- (0.2,0) -- cycle;
            \fill[\blau] (170:1) arc (170:180:1) -- (-0.2,0) -- cycle;
          \end{scope}
          \eight
        \end{scope}  
        \begin{scope}[shift={(112.5:6)}]
          \begin{scope}[rotate={180}]
            \clip\eightpath;
            \fill[\blau] (1,0) arc (0:180:1);
            \fill[white] (112.5:2) circle (1.5);
            \fill[white] (1,0) arc (0:55:1) -- (0.2,0) -- cycle;
            \fill[white] (170:1) arc (170:180:1) -- (-0.2,0) -- cycle;
          \end{scope}
          \eight
        \end{scope}  
        \begin{scope}[shift={(112.5:9)}]
          \begin{scope}[rotate={180}]
            \clip\eightpath;
            \fill[\blau] (112.5:2) circle (1.5);
            \fill[\blau] (1,0) arc (0:55:1) -- (0.2,0) -- cycle;
            \fill[\blau] (170:1) arc (170:180:1) -- (-0.2,0) -- cycle;
          \end{scope}
          \eight
        \end{scope}  
      \end{tikzpicture}
    \end{center}

    \caption{The open cover of~$\Sigma$ in Example~\ref{exa:ex2}}
    \label{fig:ex2}
  \end{figure}
  
  Then the nerve~$N$ of~$U$ satisfies~$|N| \cong \Delta^3 \simeq \bullet$.
  In particular, $H^2(|N|;\R) \cong 0$. Therefore, the comparison
  map~$H_b^2(\Sigma;\R) \longrightarrow H^2(\Sigma;\R)$ (which is non-trivial)
  cannot factor over~$H^2(|N|;\R)$.  
\end{example}

\subsection{Applications}

The standard application of results of the type of
Theorem~\ref{thm:main} and Theorem~\ref{thm:mainl1} are vanishing
theorems for the $\ell^1$-semi-norm on singular homology (and whence
to simplicial volume):

\begin{cor}[a vanishing theorem]\label{cor:vanishing}
  Let $X$ be a path-connected CW-complex, let $U$ be an amenable
  convex open cover of~$X$, and let $N$ be the nerve of~$U$.
  Furthermore, let $k \in \N$ with~$H_k(N;\R) \cong 0$.
  \begin{enumerate}
  \item Then the comparison maps~$c_X \colon H_b^k(X;\R) \longrightarrow H^k(X;\R)$
    in bounded cohomology
    and $c_X^{\ell^1} \colon H_k(X;\R) \longrightarrow H^{\ell^1}_k(X;\R)$ in $\ell^1$-homology
    are the zero maps.
  \item In particular, for all~$\alpha \in H_k(X;\R)$, we have
    \[ \|\alpha\|_1 = 0. \]
  \end{enumerate}
\end{cor}
\begin{proof}
  \emph{Ad~1.} We have~$H_k(|N|;\R) \cong H_k(N;\R) \cong 0$ by assumption.
  Moreover, by the universal coefficient theorem, we also
  have~$H^k(|N|;\R) \cong 0$. Therefore, Theorem~\ref{thm:main}
  and Theorem~\ref{thm:mainl1} show that the comparison maps
  in bounded cohomology and $\ell^1$-homology factor over~$0$
  and so are the zero maps.

  \emph{Ad~2.} The comparison map~$c^{\ell^1}_X$ is isometric with
  respect to the $\ell^1$-semi-norm~\cite[Proposition~2.5]{loehphd}.
  Applying the first part proves the claim.
\end{proof}

The hypothesis is satisfied if the multiplicity of the cover is at most~$k$
(and thus the nerve has dimension at most~$k-1$). Further examples are contained
in Gromov's original article~\cite[Section~3.1]{vbc}.

\subsection{Notation}

In the rest of this paper, homology, cohomology as well as bounded
cohomology of spaces, groups, or simplicial complexes is always taken
with $\R$-coefficients (and we will mostly omit this from the
notation).

\subsection*{Organisation of this article}

We first recall basics on bounded cohomology (Section~\ref{sec:Hb}),
classifying spaces of families of subgroups (Section~\ref{sec:EFG}),
and nerves of covers (Section~\ref{sec:nerve}). It should be noted
that all of this material is standard; we collected it here in one
place for convenience and to introduce the notation used in the main
proof.

The proof of Theorem~\ref{thm:main} is given in
Section~\ref{sec:mainproof}; the proof of Theorem~\ref{thm:mainl1} is
given in Section~\ref{sec:mainproofl1}.

\section{Preliminaries: bounded cohomology}\label{sec:Hb}

We first collect basic notation and basic facts on bounded cohomology
and $\ell^1$-homology, as needed in the sequel; in fact, no other
input from bounded cohomology will be needed for the proofs of
Theorem~\ref{thm:main} and Theorem~\ref{thm:mainl1}. For details and
further results, we refer the reader to the
literature~\cite{vbc,ivanov,ivanovTNG,monod,frigeriobc,bouarich,loehl1}.

\subsection{Bounded cohomology}

Bounded cohomology of spaces and groups is defined as the cohomology
of the topological dual of the standard chain complexes.
Let $B(\args,\R)$ be the contravariant endofunctor
on the category of normed $\R$-vector spaces and continuous linear maps
that is given by taking the topological dual.
A \emph{normed chain complex} is a chain complex consisting of normed
$\R$-vector spaces and continuous boundary maps. Then $B(\args,\R)$ 
induces a contravariant functor from the category of normed chain
complexes (and degree-wise continuous chain maps) to the category
of Banach cochain complexes (and degree-wise continuous cochain maps).

If $X$ is a topological space, then the singular chain complex~$C_*(X;\R)$
is a normed chain complex with respect to the $\ell^1$-norm~$|\cdot|_1$
associated with the $\R$-bases given by the sets of all singular simplices
of~$X$.
If $f \colon X \longrightarrow Y$ is a continuous map, then the
chain map~$C_*(f;\R) \colon C_*(X;\R) \longrightarrow C_*(Y;\R)$
is degree-wise of norm at most~$1$.

\begin{defi}[bounded cohomology of spaces]
  \hfil
  \begin{itemize}
  \item
    The \emph{bounded cochain complex functor~$C_b^*(\args;\R)$}
    is the contravariant functor from the category of topological spaces
    to Banach cochain complexes given by the composition~$B(C_*(\args;\R), \R)$.
  \item
    The \emph{bounded cohomology functor~$H_b^*(\args;\R)$}
    is the contravariant functor from the category of topological spaces
    to (semi-normed) graded $\R$-vector spaces, given by the
    composition~$H^*(C_b^*(\args;\R))$.
  \item
    The natural transformation~$H_b^*(\args;\R) \longrightarrow H^*(\args;\R)$
    induced by the natural inclusion~$C_b^*(\args;\R) \longrightarrow C^*(\args;\R)$
    is the \emph{comparison map}.
  \end{itemize}
\end{defi}

One of the key properties of bounded cohomology is Gromov's mapping
theorem~\cite[p.~40]{vbc}\cite[Theorem~4.3]{ivanov}\cite{ivanovTNG}\cite[Corollary~5.11]{frigeriobc}:

\begin{thm}[Gromov's mapping theorem]\label{thm:mthm}
  Let $f \colon X \longrightarrow Y$ be a continuous map of 
  path-connected spaces such that $\pi_1(f) \colon \pi_1(X)
  \longrightarrow \pi_1(Y)$ is surjective and has amenable kernel.
  Then $H_b^*(f;\R) \colon H_b^*(Y;\R) \longrightarrow
  H_b^*(X;\R)$ is an (isometric) isomorphism.
\end{thm}

Furthermore, bounded cohomology admits a description in terms of
injective resolutions~\cite{ivanov}\cite[Chapter~4]{frigeriobc}. We will
need the following facts:

\begin{prop}[\protect{\cite[Lemma~4.22]{frigeriobc}}]\label{prop:relinj}
  Let $\Gamma$ be a group and let $S$ be a $\Gamma$-set all of
  whose isotropy groups are amenable. Then $B(S,\R)$ is a relatively
  injective $\Gamma$-module.
\end{prop}
  
\begin{thm}\label{thm:relinjres}
  Let $\Gamma$ be a group, let $\R \longrightarrow C^*$ be a strong
  relatively injective resolution of~$\R$ by Banach $\Gamma$-modules,
  and let $X$ be a path-connected topological space with fundamental
  group~$\Gamma$. Then every degree-wise bounded $\Gamma$-cochain
  map~$C^* \longrightarrow C_b^*(\EFG{}\Gamma;\R)$ that
  extends~$\id_\R \colon \R \longrightarrow \R$ induces an
  isomorphism~$H^*((C^*)^\Gamma) \longrightarrow H^*(C_b^*(\EFG{}\Gamma)^\Gamma;\R)$.
\end{thm}
\begin{proof}
  The cochain complex~$C_b^*(\EFG{}\Gamma;\R)$, together with the
  canonical augmentation~$\R \longrightarrow C_b^0(\EFG{}\Gamma;\R)$,
  is a strong relatively injective
  resolution~\cite[Section~4]{ivanov}\cite[Proposition~4.8]{frigeriobc}. Therefore,
  applying the fundamental theorem for this type of homological
  algebra~\cite[Section~3]{ivanov}\cite[Theorem~4.5]{frigeriobc}
  completes the proof.
\end{proof}

\subsection{$\ell^1$-Homology}\label{subsec:l1homology}

Instead of taking the topological dual functor, one can also take the
completion functor. Applying the completion functor to the singular
chain complex~$C_*(\args;\R)$ (with the $\ell^1$-norm) leads to the
\emph{$\ell^1$-chain complex~$C_*^{\ell^1}(\args;\R)$} and, after
taking homology, to \emph{$\ell^1$-homology~$H_*^{\ell^1}(\args;\R)$}.

While the duality between $\ell^1$-homology and bounded cohomology
is not as straightforward as in the case of singular (co)homology,
we still have the following tools:

\begin{thm}[translation principle~\protect{\cite[Theorem~1.1]{loehl1}}]\label{thm:transl}
  If $f_* \colon C_* \longrightarrow D_*$ is a morphism of Banach
  chain complexes, then $H^*(B(f_*,\R)) \colon H^*(B(D_*,\R))
  \longrightarrow H^*(B(C_*,\R))$ is an isomorphism if and only if
  $H_*(f_*) \colon H_*(C_*) \longrightarrow H_*(D_*)$ is an
  isomorphism.
\end{thm}

\begin{cor}[mapping theorem for $\ell^1$-homology~\protect{\cite[Corollary~5.2]{loehl1}\cite[Corollaire~5]{bouarich}}]\label{cor:mthml1}
  Let $f \colon X \longrightarrow Y$ be a continuous map between
  path-connected spaces such that $\pi_1(f) \colon \pi_1(X)
  \longrightarrow \pi_1(Y)$ is surjective and has amenable
  kernel. Then $H_*^{\ell^1}(f;\R) \colon H_*^{\ell^1}(X;\R)
  \longrightarrow H_*^{\ell^1}(Y;\R)$ is an (isometric) isomorphism.
\end{cor}

\section{Preliminaries: classifying spaces of families of subgroups}\label{sec:EFG}

\subsection{Classifying spaces}

We briefly recall basic terminology concerning classifying spaces
of families of subgroups; further information can, e.g., be found
in L\"uck's survey~\cite{luecksurvey}.

\begin{defi}[subgroup family]\label{def:family}
  Let $\Gamma$ be a group. A \emph{subgroup family} of~$\Gamma$
  is a set~$F$ of subgroups of~$\Gamma$ with the following properties:
  \begin{itemize}
  \item The set~$F$ is closed under conjugation.
  \item The set $F$ is closed under taking subgroups.
  \end{itemize}
\end{defi}

\begin{defi}[classifying space of a subgroup family]\label{def:EFG}
  Let $\Gamma$ be a group and let $F$ be a subgroup family
  of~$\Gamma$.
  \begin{itemize}
  \item
    A $\Gamma$-CW-complex has \emph{$F$-restricted
      isotropy} if all isotropy groups lie in~$F$.
  \item
    A \emph{model for~$\EFG F \Gamma$} is a
    $\Gamma$-CW-complex~$X$ with $F$-restricted isotropy and the following
    universal property:
    For every $\Gamma$-CW-complex~$Y$ with $F$-restricted isotropy,  
    there exists up to $\Gamma$-homotopy exactly one
    continuous $\Gamma$-map~$Y \longrightarrow X$. 

    We will also abuse the symbol~$\EFG F \Gamma$ to denote a choice of
    a model for~$\EFG F \Gamma$ (this is well-defined up to canonical
    $\Gamma$-homotopy equivalence) and $f_{Y,\Gamma,F} \colon Y \longrightarrow \EFG F \Gamma$
    for a choice of a (``the'') continuous $\Gamma$-map.
  \end{itemize}
  If $F$ is the family that only contains the trivial subgroup of~$\Gamma$,
  then $\EFG F \Gamma = \EFG{}\Gamma$, and we abbreviate~$f_{Y,\Gamma} := f_{Y,\Gamma,F}$.
\end{defi}

\begin{thm}[alternative characterisation of classifying spaces~\protect{\cite[Theorem 1.9]{luecksurvey}}]\label{thm:EFG}
  Let $\Gamma$ be a group and let $F$ be a subgroup family of~$\Gamma$.
  \begin{enumerate}
  \item Then there exists a model for~$\EFG F \Gamma$.
  \item A $\Gamma$-CW-complex~$Y$ is a model for~$\EFG F \Gamma$ if and only if
    the following conditions are satisfied:
    \begin{itemize}
    \item The $\Gamma$-CW-complex~$Y$ has $F$-restricted isotropy.
    \item For every subgroup~$H \in F$, the $H$-fixed point set~$X^H$ is weakly
      contractible.
    \end{itemize}
  \end{enumerate}
\end{thm}

\subsection{The family of amenable subgroups}

Classically, key examples are given by the family that consists solely
of the trivial subgroup (leading to the ordinary classifying space)
and the subgroup family of all finite subgroups (leading for discrete
groups to the classifying space of proper actions). In the setting of
bounded cohomology, it is natural to work with the extension of finite
groups to amenable groups:

\begin{example}
  Let $\Gamma$ be a group and let $\Amen$ be the set of all amenable
  subgroups of~$\Gamma$. Then $\Amen$ is a subgroup family of~$\Gamma$
  in the sense of Definition~\ref{def:family}.
\end{example}


\subsection{An example}

We explain an explicit model of $\EFG \Amen \Gamma$ for the free group
$\Gamma=F_2$ of rank~$2$.  In this case, $\Amen$ is the family of
cyclic subgroups of $F_2$. The following lemma holds, suitably
modified, in the greater generality of word-hyperbolic
groups~\cite[Remark~7]{leary}.

\begin{lem}\label{lem:cyclic}
  Every cyclic subgroup of $F_2$ is contained in a maximal cyclic
  subgroup. The normaliser of a maximal cyclic subgroup $C$ of $F_2$
  is $C$ itself.
\end{lem}

If $K, H<\Gamma$ are subgroups of a group~$\Gamma$ and $X$ is an
$H$-space, then the $K$-fixed points of the induced $\Gamma$-space
$\Gamma\times_H X$ are
\begin{equation}\label{eq:fixed points}
 \bigl( \Gamma\times_H X\bigr)^K=\bigl\{ [\gamma, x]\mid \gamma^{-1}K\gamma\subset H,~x\in X^{\gamma^{-1}K\gamma}\bigr\}.
\end{equation}

\begin{example}[a model of~$\EFG \Amen{F_2}$]
Let $\Mcyc$ be a complete set of representatives of conjugacy classes
of maximal cyclic subgroups in $F_2$. Every subgroup $C\in \Mcyc$ is
isomorphic to $\Z$. Hence we can take $\R$ as a model of $\EFG {}C$
for every $C\in\Mcyc$ on which $C\cong\Z$ acts by translations. We
pick the $4$-regular tree $T$ as a model of $\EFG {} F_2$.  A
$2$-dimensional model $Y$ of the classifying space $\EFG\Amen F_2$ is
given by the pushout of $F_2$-spaces:
\[ \xymatrix{
 \coprod_{C\in\Mcyc} F_2\times_C \R\ar[r]^-{c}\ar[d] &
 T\ar[d]\\
 \coprod_{C\in\Mcyc} F_2\times_C \cone(\R)\ar[r] & Y }\]
Here $\cone(\R)$ is the cone over the free $C$-space~$\R$. The
$C$-action on~$\R$ naturally extends to the cone in such a way that
the cone tip is a fixed point. The map $c$ is the classifying map for
$T$ as a model of $\EFG {} F_2$. The left vertical map is induced by
the inclusion of the bottom into the cone.

According to Theorem~\ref{thm:EFG} we have to show that $Y^C$ is
contractible for every cyclic subgroup $C<F_2$ and is empty for every
non-cyclic subgroup $C<F_2$. Let $K<F_2$ be a non-trivial subgroup. We
obviously have $T^K=\emptyset$ and by~\eqref{eq:fixed points} also
\[ \biggl(\coprod_{C\in\Mcyc} F_2\times_C \R\biggr)^K=\coprod_{C\in\Mcyc} \bigl(F_2\times_C \R\bigr)^K=\emptyset.\]
Since taking $K$-fixed points respects the pushout
property~\cite[(1.17) exercise~5 on p.~103]{tomdieck} we obtain that
\[ Y^K\cong \biggl(\coprod_{C\in\Mcyc} F_2\times_C \cone(\R)\biggr)^K=\coprod_{C\in\Mcyc} \bigl(F_2\times_C \cone(\R)\bigr)^K.\]
If $K$ is not cyclic, then it follows from~\eqref{eq:fixed points}
that $Y^K$ is empty. Let $K$ be cyclic. Let $C_0\in\Mcyc$ be the
unique element such that a conjugate $\gamma_0 C_0 \gamma_0^{-1}$,
$\gamma_0\in F_2$, is the unique maximal cyclic subgroup containing
$K$ (Lemma~\ref{lem:cyclic}). Further, $\gamma_0$ is uniquely
determined up to multiplication with elements in the normaliser of
$C_0$, which equals $C_0$.  Then~\eqref{eq:fixed points} implies that
\[ Y^K=\bigl(F_2\times_{C_0} \cone(\R)\bigr)^K=\{[\gamma_0, \text{cone tip}]\}\]
consists of a single point. It remains to show that $Y^K=Y$ is
contractible for $K=\{1\}$. We only sketch the argument: The three
spaces defining $Y$ in the above pushout have contractible path
components. Hence $Y$ is acyclic. Using the van Kampen theorem one
verifies that $Y$ is simply connected. Whitehead's theorem then
implies that $Y$ is contractible.
\end{example}

The above example can be generalized to word-hyperbolic
groups~\cite{leary}.

\section{Preliminaries: nerves of covers}\label{sec:nerve}

In the following, we discuss nerves of open covers as well as their
lifts to universal coverings. In order to keep the notation simple, we
will view open covers as \emph{sets} of open subsets of the given
ambient space, not as families of subsets.

\begin{setup}\label{set:cov}
  Let $X$ be a path-connected CW-complex with 
  universal covering~$\pi \colon \ucov X \longrightarrow X$, let $x_0
  \in X$, and let~$\Gamma := \pi_1(X,x_0)$ be the fundamental group of~$X$.
  Moreover, let $U$ be an open cover of~$X$ by path-connected sets
  and
  let
  \begin{align*}
    \widetilde U := \bigl\{ V \subset \widetilde X \bigm| \;
    & \text{there exists a~$W \in U$ such that}
    \\
    & \text{$V$ is a path-connected component of~$\pi^{-1}(W)$} \bigr\}
  \end{align*}
  be the associated cover of~$\ucov X$.
\end{setup}

\begin{rem}\label{rem:projlift}
  In the situation of Setup~\ref{set:cov}, let $V \in \widetilde U$.
  Then $V$ is open (as $X$ is locally path-connected and
  standard lifting properties in coverings apply).

  Moreover, $\pi(V) \in U$:
  By construction, there is a~$W \in U$ such that $V$ is a path-connected
  component of~$\pi^{-1}(W)$. In particular, $\pi(V) \subset W$. 
  In fact, we have~$\pi(V) = W$. Indeed, let $x \in W$. Because $V$ is non-empty and $W$ is
  path-connected, there is a continuous path~$w \colon [0,1]
  \longrightarrow W$ with~$w(0) \in \pi(V)$ and~$w(1) = x$.  Then the
  lifting properties of the covering~$\pi|_{\pi^{-1}(W)}
  \colon \pi^{-1}(W) \longrightarrow W$ show that there is a continuous
  $\pi$-lift~$\widetilde w \colon [0,1] \longrightarrow \pi^{-1}(W)$
  with~$\widetilde w(0) \in V$. Because $\widetilde w([0,1])$ is
  path-connected and $V$ is a path-connected component
  of~$\pi^{-1}(W)$, we obtain~$\widetilde w(1) \in V$. In particular,
  \[ x = w(1) = \pi \circ \widetilde w(1) \in \pi(V).
  \]
\end{rem}

\begin{defi}[convex cover]\label{def:convex}
  An open cover~$U$ of a topological space~$X$ is called \emph{convex} if
  for every finite set~$U' \subset U$, the intersection~$\bigcap U'$ is
  path-connected (or empty).
\end{defi}

\subsection{Nerves and equivariance}

\begin{defi}[nerve]
  In the situation of Setup~\ref{set:cov}, the \emph{nerve of~$U$} is
  the (abstract) simplicial complex~$N$ given by the following data:
  For each~$n \in \N$, the set of $n$-simplices of~$N$ is
  \[ \biggl\{ \{V_0, \dots, V_n\}
  \biggm| V_0, \dots, V_n \in U,\ 
     \bigcap_{j=0}^n V_j \neq \emptyset,\
  \fa{j,k \in \{0,\dots,n\}, j\neq k} V_j \neq V_k \biggr\}.
  \]
\end{defi}

\begin{lem}[actions on nerves]\label{lem:nerve}
  In the situation of Setup~\ref{set:cov}, let $N$ be the nerve of~$U$
  and let $\widetilde N$ be the nerve of~$\widetilde U$. Then:
  \begin{enumerate}
  \item
    The deck transformation action of~$\Gamma$ on~$\widetilde X$ turns
    the nerve~$\widetilde N$ of the induced cover~$\widetilde U$ into
    a $\Gamma$-simplicial complex.
  \item
    The universal covering map~$\pi \colon \widetilde X \longrightarrow X$
    induces a well-defined simplicial map~$p \colon \widetilde N \longrightarrow N$.
  \item
    If the open cover~$U$ is convex and $n \in \N$, then $p$ induces a
    bijection between~$\Gamma \setminus (\widetilde N)_n$ and~$N_n$.
    Here, $(\widetilde N)_n$ carries the diagonal $\Gamma$-action.
  \end{enumerate}
\end{lem}
\begin{proof}
  \emph{Ad~1.}
  By construction, the set $\widetilde U$ is closed under the
  deck transformation action of~$\Gamma$. Morover, this $\Gamma$-action
  is compatible with the simplicial structure (because homeomorphisms
  preserve intersections).
  
  \emph{Ad~2.} For every~$V \in \widetilde
  U$, we have~$\pi(V) \in U$ (Remark~\ref{rem:projlift}).

  Let $n \in \N$ and let $\{V_0, \dots, V_n\}$ be an $n$-simplex of~$\widetilde N$.
  Then, the projections~$\pi(V_0), \dots, \pi(V_n)$ are pairwise different
  (because $V_0 \cap \dots \cap V_n \neq \emptyset$ and elements of~$\widetilde U$
  that lie over the same set in~$U$ have empty intersection). 
  Moreover,
  \[ \bigcap_{j=0}^n \pi(V_j) \supset \pi\biggl( \bigcap_{j=0}^n V_j \biggr)
     \neq \emptyset.
  \]
  Hence, $n$-simplices are mapped to $n$-simplices.

  \emph{Ad~3.}
  Let $n \in \N$ and let $\{W_0, \dots, W_n\}$ be an
  $n$-simplex of~$N$.
  \begin{itemize}
  \item
    Then there exists an $n$-simplex~$\{V_0, \dots, V_n\}$
    of~$\widetilde N$ with
    \[ p\bigl(\{V_0, \dots, V_n \}\bigr) = \{W_0,\dots,W_n\}.
    \]
    This can be seen as follows: Let $x \in \bigcap_{j=0}^n W_j$ and
    let $\widetilde x \in \pi^{-1}(\{x\})$. Then, for each~$j \in \N$,
    we choose the element~$V_j \in \widetilde U$ with~$\widetilde x \in V_j$
    and $\pi(V_j) = W_j$. By construction, the intersection~$\bigcap_{j=0}^n V_j$
    contains~$\widetilde x$ and thus is non-empty. Therefore, $\{V_0, \dots, V_n\}$
    is an $n$-simplex of~$\widetilde N$.
  \item
    If $\{V'_0, \dots, V'_n\}$ is another $n$-simplex of~$\widetilde N$
    with~$p(\{V'_0, \dots, V'_n\}) = \{W_0,\dots, W_n\}$, then there
    exists a~$\gamma \in \Gamma$ with
    \[ \fa{j \in \{0,\dots, n\}} V_j = \gamma \cdot V'_j,
    \]
    because:
    Let $x \in \bigcap_{j=0}^n V_j$ and $y \in \bigcap_{j=0}^n V'_j$. As $U$
    is a convex open cover, the intersection~$\bigcap_{j=0}^n W_j$ is path-connected.
    Let $w \colon [0,1] \longrightarrow \bigcap_{j=0}^n W_j$ be a path from~$\pi(x)$
    to~$\pi(y)$ and let $\widetilde w \colon [0,1] \longrightarrow \widetilde X$
    be a $\pi$-lift of~$w$. Then $\widetilde w([0,1]) \subset V_j$ for each~$j \in \{0,\dots, n\}$
    and there exists a~$\gamma \in \Gamma$ with
    \[ \gamma \cdot y  = \widetilde w(1).
    \]
    By construction, $\widetilde w(1) \in \bigcap_{j=0}^n \gamma \cdot V'_j$. Therefore,
    for all~$j \in \{0,\dots,n\}$, we have~$V_j \cap \gamma \cdot V'_j \neq \emptyset$
    and so $V_j = \gamma \cdot V'_j$.
    \qedhere
  \end{itemize}
%
%
\end{proof}

\begin{prop}\label{prop:C*p}
  In the situation of Setup~\ref{set:cov}, let $U$ be convex,
  let $N$ be the nerve of~$U$, 
  and let $\widetilde N$ be the nerve of~$\widetilde U$.
  Then the map~$p \colon \widetilde N \longrightarrow N$ induced
  by~$\pi$ (Lemma~\ref{lem:nerve}) induces a chain homotopy equivalence
  \[ C_*(|p|)_\Gamma \colon C_*(|\widetilde N|)_\Gamma
     \longrightarrow C_*(|N|).
  \]
\end{prop}
\begin{proof}
  By the previous lemma (Lemma~\ref{lem:nerve}), the simplicial map~$p$
  induces a chain isomorphism~$C_*(\widetilde N)_\Gamma \longrightarrow C_*(N)$
  between the \emph{simplicial} chain complexes.

  Moreover, the canonical inclusion~~$i \colon C_*(N) \longrightarrow
  C_*(|N|)$ is a chain homotopy equivalence and the canonical
  inclusion~$\widetilde i \colon C_*(\widetilde N) \longrightarrow
  C_*(|\widetilde N|)$ is a $\Gamma$-chain homotopy
  equivalence~\cite[Proposition~13.10~b) on p.~264]{lueck}. It should
  be noted that this is the step in the proof of
  Theorem~\ref{thm:main}, where we lose control over the norms.

  Therefore, the commutative diagram
  \[ \xymatrix{%
    C_*(|\widetilde N|)_\Gamma
    \ar[r]^-{C_*(|p|)_\Gamma}
    & C_*(|N|)
    \\
    C_*(\widetilde N)_\Gamma
    \ar[r]_-{C_*(p)_\Gamma}^-{\simeq}
    \ar[u]^{\widetilde i_\Gamma}_{\simeq}
    & C_*(N)
    \ar[u]_{i}^{\simeq}
    }
  \]
  proves the claim.
\end{proof}

\subsection{The nerve map}

\begin{rem}[nerve map~\protect{\cite[p.~355]{dold}}]
  In the situation of Setup~\ref{set:cov}, the space~$X$
  admits a partition of unity subordinate to the open cover~$U$
  (because CW-complexes are Hausdorff and paracompact). Every partition of
  unity~$(\varphi_V)_{V \in U}$ of~$X$ that is subordinate to~$U$
  gives rise to a continuous map
  \[ \nu \colon X \longrightarrow |N|
  \]
  into the geometric realisation~$|N|$ of the nerve~$N$ of~$U$:
  For~$x \in X$, we just set (in barycentric coordinates) 
  \[ \nu(x) := \sum_{V \in U} \varphi_V(x) \cdot V.
  \]
  Different choices of partitions of unity lead to homotopic maps. We
  therefore speak of \emph{the nerve map} $X\to |N|$.
\end{rem}

\begin{lem}[lifting the nerve map]\label{lem:nervemaplift}
  In the situation of Setup~\ref{set:cov}, let $N$ be the nerve
  of~$U$, let $\widetilde N$ be the nerve of~$\widetilde U$, let
  $|p| \colon |\widetilde N| \longrightarrow |N|$ be the map induced
  by~$\pi$ (Lemma~\ref{lem:nerve}), and let $\nu \colon X \longrightarrow |N|$
  be a nerve map.
  Then there exists a continuous $\Gamma$-map~$\widetilde
  \nu \colon \ucov X \longrightarrow |\widetilde N|$ with
  \[ |p| \circ \widetilde \nu = \nu \circ \pi.
  \]
  \[ \xymatrix{%
    \widetilde X \ar@{-->}[r]^-{\widetilde \nu} \ar[d]_-{\pi}
    & |\widetilde N|\ar[d]^-{|p|}
    \\
    X \ar[r]_-{\nu}
    & |N|
    }
  \]
\end{lem}
\begin{proof}
  Let $(\varphi_W)_{W \in U}$ be a partition of unity of~$X$ that is
  subordinate to~$U$ (which induces the nerve map~$\nu$). We construct~$\widetilde \nu$
  as nerve map of the lift of this partition of unity to~$\widetilde U$:
  For~$V \in \widetilde U$, we define
  \[ \widetilde \varphi_V := \chi_V \cdot \varphi_{\pi(V)} \circ \pi
     \colon \widetilde X \longrightarrow [0,1],
  \]
  where $\chi_V \colon \widetilde X \longrightarrow \{0,1\}$ denotes
  the characteristic function of the subset~$V \subset \widetilde X$.
  We will now establish the following properties of these functions: 
  \begin{enumerate}
  \item The function~$\widetilde\varphi_V \colon \widetilde X \longrightarrow [0,1]$
    is continuous.

    [Because $X$ is locally path-connected, the standard lifting properties
      show that $\pi(\partial V) = \partial (\pi(V))$. Now continuity of~$\widetilde \varphi_V$
    easily follows from the continuity of~$\varphi_{\pi(V)}$.]
  \item For each~$\gamma \in \Gamma$, we have
    $\widetilde \varphi_V (\gamma \cdot x) = \widetilde \varphi_{\gamma^{-1} \cdot V}(x)$.

    [This equivariance property follows directly from the construction.]
  \item The family~$(\widetilde\varphi_V)_{V \in \widetilde U}$ is a partition
    of unity on~$\widetilde X$, subordinate to~$\widetilde U$.

    [This is a straightforward computation.]
  \end{enumerate}
  If~$W \in
  U$, then we abbreviate~$\widetilde U|_W := \{ V \in \widetilde U
  \mid \pi(V) = W \}$.  Then 
  \begin{align*}
    \widetilde \nu \colon \widetilde X & \longrightarrow |\widetilde N|
    \\
    x & \longmapsto \sum_{W \in U} \sum_{V \in \widetilde U|_W} \widetilde\varphi_V(x) \cdot V
  \end{align*}
  is the nerve map associated with this partition of
  unity~$(\widetilde\varphi_V)_{V \in \widetilde U}$. The second property
  shows that
  \begin{align*}
    \widetilde \nu(\gamma \cdot x)
    & = \sum_{W\in U} \sum_{V \in \widetilde U|_W} \widetilde \varphi_V(\gamma \cdot x) \cdot V
    \\
    & = \sum_{W\in U} \sum_{V \in \widetilde U|_W} \widetilde \varphi_{\gamma^{-1}\cdot V}(x) \cdot V
    \\
    & = \sum_{W\in U} \sum_{V \in \widetilde U|_W} \widetilde \varphi_{V}(x) \cdot (\gamma \cdot V)
    \\
    & = \gamma \cdot \widetilde \nu(x)
  \end{align*}
  holds for all~$x \in \widetilde X$ and all~$\gamma \in
  \Gamma$. Hence, $\widetilde \nu$ is a continuous $\Gamma$-map.

  Moreover, by construction, we have that $|p| \circ \widetilde \nu = \nu \circ \pi$.
\end{proof}

In summary, we can pass from the nerve~$N$ to the $\Gamma$-equivariant
setting (and whence to the right context for classifying spaces) as follows:

\begin{prop}\label{prop:transfer}
  In the situation of Setup~\ref{set:cov}, let $U$ be convex, let $N$ be the nerve
  of~$U$, let $\widetilde N$ be the nerve of~$\widetilde U$, let
  $p \colon \widetilde N \longrightarrow N$ be the map induced
  by~$\pi$ (Lemma~\ref{lem:nerve}), let $\nu \colon X \longrightarrow |N|$
  be a nerve map, and let $\widetilde \nu\colon \widetilde X \longrightarrow |\widetilde N|$ 
  be as in Lemma~\ref{lem:nervemaplift}. 
  Then there exists a chain map~$\tau_p \colon C_*(|N|) \longrightarrow
  C_*(|\widetilde N|)_\Gamma$ such that
  \[ \tau_p \circ C_*(\nu) \circ C_*(\pi)_\Gamma
  \simeq C_*(\widetilde \nu)_\Gamma
  \colon C_*(\widetilde X)_\Gamma \longrightarrow C_*(|\widetilde N|)_\Gamma.
  \]
\end{prop}
\begin{proof}
  By Lemma~\ref{lem:nervemaplift}, we have~$|p|\circ \widetilde \nu
  = \nu \circ \pi$, whence
  \[ C_*(|p|)_\Gamma \circ C_*(\widetilde \nu)_\Gamma
     = C_*(\nu) \circ C_*(\pi)_\Gamma.
  \]
  In view of Proposition~\ref{prop:C*p}, we can now take
  a chain homotopy inverse of~$C_*(|p|)_\Gamma \colon C_*(|\widetilde N|)_\Gamma
  \longrightarrow C_*(|N|)$ for~$\tau_p$.
\end{proof}

\subsection{Families and covers}

\begin{defi}[amenable cover]\label{def:amcover}
  In the situation of Setup~\ref{set:cov} (which includes
  that all sets in~$U$ are path-connected), let $F$ be a subgroup family of~$\Gamma$.
  We call $U$ an \emph{$F$-cover}
  of~$X$ if for each~$V \in U$ and each~$x \in V$, the subgroup
  \[ \im \bigl(\pi_1(V,x)\longrightarrow \pi_1(X,x) \bigr) \subset \pi_1(X,x)
  \]
  induced by the inclusion~$V \hookrightarrow X$ 
  lies in~$F$ under an isomorphism~$\pi_1(X,x) \cong \pi_1(X,x_0) = \Gamma$
  induced by conjugation with a path between~$x$ and~$x_0$
  (because $F$ is closed under conjugation in~$\Gamma$, this property does
  \emph{not} depend on the chosen paths).

  We call $U$ an \emph{amenable cover} if $U$ is an $\Amen$-cover.
\end{defi}

\begin{lem}[nerves of amenable covers]\label{lem:nerveF}
  In the situation of Setup~\ref{set:cov}, let $F$ be a subgroup family of~$\Gamma$
  and let $U$ be an $F$-cover.
  Then the isotropy groups of the corresponding
  $\Gamma$-space~$|\widetilde N|$ all lie in~$F$.

  In particular, if $F = \Amen$, then these isotropy groups are amenable. 
\end{lem}
\begin{proof}
  Because the $\Gamma$-action on~$|\widetilde N|$ (Lemma~\ref{lem:nerve}) is
  obtained from the simplicial $\Gamma$-action on~$\widetilde N$ by affine
  extension, it suffices to show that the isotropy groups of the
  vertices in the barycentric subdivision~$S$ of~$\widetilde N$ with
  the induced simplicial $\Gamma$-action all lie in~$F$.

  By definition of the barycentric subdivision, the vertex set of~$S$
  is the set of simplices of~$N$. Let $v$ be a vertex of~$S$; i.e.,
  there exist~$n \in \N$, $V_0, \dots, V_n \in \widetilde U$ with~$V_0
  \cap \dots \cap V_n \neq \emptyset$ and $v = \{V_0, \dots,
  V_n\}$. Then the stabiliser~$\Gamma_v$ of~$v$ consists precisely of
  those~$\gamma \in \Gamma$ with
  \[ \{ \gamma \cdot V_0, \dots, \gamma \cdot V_n\}
     = \{V_0, \dots, V_n\}.
  \]
  We distinguish the following cases:
  \begin{itemize}
  \item If $n = 0$, then the stabiliser of~$v$ is
    \[ \{ \gamma \in \Gamma \mid \gamma \cdot V_0 = V_0 \},
    \]
    which is (conjugate to) a subgroup of~$\im (\pi_1(\pi(V_0) \hookrightarrow X))$.
    Because $\pi(V_0) \in U$ (Remark~\ref{rem:projlift}) and $U$ is an $F$-cover,
    the stabiliser of~$v$ lies in~$F$. 
  \item Let $n > 0$. If $\gamma \in \Gamma$ is in the stabiliser
    of~$v$ and $j \in \{0,\dots, n\}$ with~$\gamma \cdot V_j = V_k$,
    then $j=k$, which can be seen as follows: Because of~$\gamma \cdot
    V_j = V_k$, we have~$\pi(V_j) = \pi(V_k)$.  Therefore,
    $V_j$ and $V_k$ are path-connected components of the
    $\pi$-preimage of the same element of~$U$ (Remark~\ref{rem:projlift}).
    On the other hand, $V_j \cap V_k \neq
    \emptyset$.  Therefore, $V_j = V_k$, and so~$j = k$.

    In particular, the stabiliser~$\Gamma_v$ is a subgroup
    of~$\Gamma_{V_0} \cap \dots \cap \Gamma_{V_n}$. The first case
    shows that $\Gamma_{V_0} \in F$. As $F$ is a
    subgroup family, also the subgroup~$\Gamma_v$ lies in~$F$.
    \qedhere
  \end{itemize}
\end{proof}

\section{Proof of Theorem~\ref{thm:main}}\label{sec:mainproof}

For the proof of Theorem~\ref{thm:main}, we will first recall
that the family of amenable subgroups can be used to compute bounded
cohomology; more generally, we introduce the notion of \hbadmissible\ 
subgroup families (and then show that $\Amen$ is such a family).

As second step, we will combine $H_b^*$-admissibility with the universal
property of classifying spaces of families to obtain the factorisation
over the cohomological nerve map.

\subsection{Admissible families of subgroups}\label{subsec:admi}

\begin{defi}
  Let $\Gamma$ be a group, let $F$ be a subgroup family of~$\Gamma$. 
  We consider the induced map 
  \[ H^*(C_b^*(f_{\EFG {} \Gamma,\Gamma,F})^\Gamma) \colon
     H^*\bigl(C_b^*(\EFG F \Gamma)^\Gamma\bigr)
     \longrightarrow H_b^*\bigl(C_b^*(\EFG {}\Gamma)^\Gamma\bigr) =: H_b^\ast(\Gamma). 
  \]
  The notation~$f_{\EFG {}\Gamma, \Gamma, F}$ is
  introduced in Definition~\ref{def:EFG}. The family~$F$ is
  \begin{itemize}
  \item  \emph{\hbadmissible}\ if $H^*(C_b^*(f_{\EFG {} \Gamma,\Gamma,F})^\Gamma)$ is surjective. 
    \item \emph{strongly \hbadmissible}\ if $H^*(C_b^*(f_{\EFG {} \Gamma,\Gamma,F})^\Gamma)$ is bijective. 
\end{itemize}
 \end{defi}
  
Clearly, every subfamily of an \hbadmissible\ family is also \hbadmissible.
Moreover, if $\Gamma$ is a group with $H_b^*(\Gamma) \cong H_b^*(1)$, then
every subgroup family of~$\Gamma$ is \hbadmissible\ (for trivial reasons);
examples of such groups are all mitotic groups~\cite{loehbcd}.

\begin{prop}\label{prop:amadmissible}
  Let $\Gamma$ be a group. Then the family~$\Amen$ of amenable
  subgroups of~$\Gamma$ is strongly \hbadmissible.
\end{prop}
\begin{proof}
  The isotropy groups of the set of singular $n$-simplices of $\EFG
  \Amen \Gamma$ are amenable since the isotropy group of $\sigma\colon
  \Delta^n\to \EFG \Amen \Gamma$ is the intersection of isotropy
  groups of points in the image of $\sigma$.  Hence $C_b^n(\EFG \Amen
  \Gamma)$ is a relatively injective Banach $\Gamma$-module for every
  $n\ge 0$ (Proposition~\ref{prop:relinj}).  Since $\EFG \Amen \Gamma$
  is also contractible (Theorem~\ref{thm:EFG}), $C_b^*(\EFG \Amen
  \Gamma)$, together with the canonical augmentation~$\R
  \longrightarrow C_b^0(\EFG \Amen \Gamma)$ by constant functions, is
  a strong relatively injective resolution of~$\R$ by Banach
  $\Gamma$-modules.

  Therefore, the fundamental theorem for this type of homological
  algebra (Theorem~\ref{thm:relinjres}) shows that the bounded
  $\Gamma$-cochain map
  \[C_b^*(f_{\EFG {} \Gamma,\Gamma,F})^\Gamma\colon C_b^*(\EFG \Amen \Gamma)^\Gamma
  \longrightarrow C_b^*(\EFG {}\Gamma)^\Gamma\] induces an isomorphism
  in bounded cohomology.  In fact, this isomorphism is even
  isometric~\cite[Theorem~4.23]{frigeriobc}.
\end{proof}

This proposition is the ``group-theoretic essence'' of
Theorem~\ref{thm:main}.

\subsection{Bounded cohomology of admissible covers}\label{subsec:admiHb}

Using the notion of $H_b^*$-admissibility, we have the following version
of Theorem~\ref{thm:main}, leading to a more conceptual understanding
of this phenomenon:

\begin{thm}\label{thm:admissible}
  Let $\Gamma$ be a group and let $F$ be an \hbadmissible\ family of
  subgroups of~$\Gamma$. Let~$X$ be a path-connected CW-complex
  with~$\pi_1(X) \cong \Gamma$, and let $U$ be an open $F$-cover
  of~$X$. Then the following hold:
  \begin{enumerate}
  \item If $U$ is convex, there exists a 
  factorisation
  \[ \xymatrix{%
    H_b^*(X;\R) \ar[r]^-{c_X} \ar@{-->}[dr]
    & H^*(X;\R)
    \\
    & H^*(|N|;\R) \ar[u]_{H^*(\nu;\R)}
    }
  \]
  of the comparison map~$c_X$ 
  through the nerve map~$\nu \colon X \longrightarrow |N|$ of~$U$.
\item If the multiplicity of $U$ is at most~$m$, then $c_X$ vanishes in all
  degrees~$\ast\ge m$. 
\end{enumerate}
\end{thm}

\begin{proof}
  We pick a basepoint~$x_0 \in X$ and consider the group~$\Gamma := \pi_1(X,x_0)$.
  Now we can invoke the classifying spaces in the following way:
  \begin{itemize}
  \item As the universal covering~$\widetilde X$ of~$X$ is a
    $\Gamma$-CW-complex with free $\Gamma$-action (via the deck
    transformation action), there is a continuous
    $\Gamma$-map~$f_{\widetilde X,\Gamma} \colon \widetilde X \longrightarrow \EFG{}\Gamma$.
  \item Because $U$ is an open $F$-cover , the geometric
    realisation~$|\widetilde N|$ of the associated open cover
    of~$\widetilde X$ is a $\Gamma$-CW-complex with $F$-restricted
    isotropy (Lemma~\ref{lem:nerveF}). Therefore, we obtain a
    corresponding continuous $\Gamma$-map~$f_{|\widetilde N|,\Gamma,F} \colon
    |\widetilde N| \longrightarrow \EFG F \Gamma$.
  \item Moreover, we have the continuous $\Gamma$-map~$f_{\EFG {} \Gamma,\Gamma,F}
    \colon \EFG {}\Gamma \longrightarrow \EFG F \Gamma$.
  \end{itemize}
  \emph{Ad~1}. To this end, we consider the diagram in
  Figure~\ref{fig:THEdiag} and explain why it commutes up to cochain
  homotopy. The squares on the left-hand side clearly commute.
  
  The universal property of~$\EFG F \Gamma$ implies that the rectangle
  in the middle commutes up to cochain homotopy: Both~$f_{\EFG
    {}\Gamma,\Gamma,F} \circ f_{\widetilde X,\Gamma}$ and
  $f_{|\widetilde N|,\Gamma,F} \circ \widetilde \nu$ are $\Gamma$-equivariant
  continuous maps~$\widetilde X \longrightarrow \EFG F \Gamma$, which
  thus have to be $\Gamma$-homotopic. 
  The right polygon is commutative by (the algebraic dual of)
  Proposition~\ref{prop:transfer}.

  Taking cohomology leads to the solid part of the following
  commutative diagram:
  \[ \xymatrix{%
    {\tikz \coordinate (dom);}H_b^*(X)     \ar[r]^-{c_X}
    \ar[d]_{\cong}
    & H^*(X)
    \ar@{=}[r]
    \ar[d]
    & H^*(X)
    \\
    H_b^*(\Gamma)
    \ar[r]_-{c_\Gamma}
    \ar@{-->}[dr]
    &
    H^*(\Gamma)
    \\
    &
    H^*(C^*(\EFG F \Gamma)^\Gamma)
    \ar[r]
    \ar[u]_{H^*(f^\Gamma)}
    & H^*{\tikz \coordinate (target);}(|N|)
    \ar[uu]_{H^*(\nu)}
  }
   \begin{tikzpicture}[overlay]
  \path[->,dashed] (dom)++(0pt,-3pt) edge [bend right=100] node [below] {$\varphi$} (target) ;  \end{tikzpicture}
  \]\vspace{0.5cm}
  
  The left vertical arrow exists and is an isomorphism by Gromov's mapping
  theorem (Theorem~\ref{thm:mthm}) 
  (because~$f_{\widetilde X,\Gamma}$ induces
  an isomorphism on bounded cohomology). 
  Moreover, because $F$ is \hbadmissible, we can easily fill in the dashed
  diagonal arrow to obtain a commutative triangle.
  Hence, taking the dotted composition gives us the claimed 
  factorisation~$\varphi$.\smallskip\\
  \emph{Ad~2}. Without the convexity assumption we still have a
  similar commutative diagram as above with the right vertical
  map~$H^\ast(\nu)$ replaced by
  \[ H^\ast(\pi)^{-1}\circ H^\ast(\widetilde \nu)\colon H^\ast(C^\ast(|\widetilde N|)^\Gamma)\rightarrow H^\ast(X).\]
  
  Under the assumption on multiplicity, the dimension of the
  simplicial complex $\widetilde N$ satisfies
  \[ \dim \widetilde N\le \dim N< m.\]
  Furthermore, the simplicial chain complex $C_\ast(\widetilde N)$ and
  the singular chain complex $C_\ast(|\widetilde N|)$ are
  equivariantly chain homotopic~\cite[Proposition~13.10~b) on
    p.~264]{lueck}. Hence $H^\ast(C^\ast(|\widetilde N|)^\Gamma)$
  vanishes in degrees~$*\ge m$ and statement~(2) follows.
 \end{proof}

\begin{figure}[h]
  \begin{center}
    $\xymatrix{%
      C_b^*(X)
      \ar@{^{(}->}[r]
      \ar[d]_{\pi}^{\cong}
      &
      C^*(X)
      \ar@{=}[rr]
      \ar[d]^{\pi}_{\cong}
      &
      &
      C^*(X)
      \\
      C_b^*(\widetilde X)^\Gamma
      \ar@{^{(}->}[r]
      &
      C^*(\widetilde X)^\Gamma
      \ar@{=}[r]
      &
      C^*(\widetilde X)^\Gamma
      &
      \\
      C_b^*(\EFG {} \Gamma)^\Gamma
      \ar@{^{(}->}[r]
      \ar[u]^{f_{\widetilde X,\Gamma}}
      &
      C^*(\EFG {} \Gamma)^\Gamma
      \ar[u]_{f_{\widetilde X,\Gamma}}
      &
      &
      \\
      & 
      C^*(\EFG F \Gamma)^\Gamma
      \ar[u]_{f_{\EFG {} \Gamma, \Gamma,F}}
      \ar[r]_-{f_{|\widetilde N|,\Gamma,F}}
      &
      C^*(|\widetilde N|)^\Gamma
      \ar[r]_-{\Hom_\R(\tau_p,\R)}^-{\simeq}
      \ar[uu]^{\widetilde \nu}
      &
      C^*(|N|)
      \ar[uuu]_{\nu}
      \\
      &
      &
      \\
      }
    $
  \end{center}

  \caption{Proof of Theorem~\ref{thm:admissible}; for improved
    readability, cochain maps induced by continuous map only
    carry the underlying continuous map as label.}
  \label{fig:THEdiag}
\end{figure}

The proof shows that we need even less than $H_b^*$-admissibility: it
suffices that the composition~$H_b^*(\Gamma;F) \longrightarrow H_b^*(\Gamma)
\longrightarrow H^*(\Gamma)$ of the canonical map with the comparison
map hits all bounded classes in~$H^*(\Gamma)$.

\subsection{Proof of Theorem~\ref{thm:main}}

As a special case, we obtain Theorem~\ref{thm:main}: By
Proposition~\ref{prop:amadmissible}, the family~$\Amen$ of amenable
subgroups is \hbadmissible. Therefore, Theorem~\ref{thm:admissible}
is applicable.

\section{The case of $\ell^1$-homology}\label{sec:mainproofl1}

We now explain how to derive the $\ell^1$-analogues of
Proposition~\ref{prop:amadmissible} and Theorem~\ref{thm:admissible},
and whence prove Theorem~\ref{thm:mainl1}.

\subsection{$\ell^1$-Admissiblity}

\begin{defi}
  Let $\Gamma$ be a group, let $F$ be a subgroup family of~$\Gamma$,
  and let $f_{\EFG{}\Gamma,\Gamma,F} \colon \EFG{}\Gamma \longrightarrow \EFG F\Gamma$
  be the canonical map; we consider the induced map 
  \[ H_*(C_*^{\ell^1}(f_{\EFG{}\Gamma,\Gamma,F})_\Gamma) \colon
    H^{\ell^1}_*(\Gamma) := H_*\bigl(C^{\ell^1}_*(\EFG{}\Gamma)_\Gamma\bigr)
    \longrightarrow H_*\bigl(C^{\ell^1}_*(\EFG F\Gamma)_\Gamma).\]
  The family~$F$ is
  \begin{itemize}
  \item \emph{\ladmissible}\ if $H_*(C_*^{\ell^1}(f_{\EFG{}\Gamma,\Gamma,F})_\Gamma)$
    is injective.
  \item \emph{strongly \ladmissible}\ if $H_*(C_*^{\ell^1}(f_{\EFG{}\Gamma,\Gamma,F})_\Gamma)$
    is bijective.
  \end{itemize}
\end{defi}

\begin{prop}\label{prop:amadmissiblel1}
  \hfil
  \begin{enumerate}
  \item
      Let $\Gamma$ be a group and let $F$ be a subgroup family.
      If $F$ is strongly \hbadmissible, then $F$ is strongly \ladmissible.
  \item In particular, the family~$\Amen$ is \ladmissible.
  \end{enumerate}
\end{prop}
\begin{proof}
  The first part follows from the fact that the topological dual of
  the topological coinvariants are the invariants of the topological dual
  and the translation principle (Theorem~\ref{thm:transl}). 
  The second part is then a consequence of the first part and
  Proposition~\ref{prop:amadmissible}.
\end{proof}

\subsection{$\ell^1$-Homology of admissible covers}

\begin{thm}\label{thm:admissiblel1}
  Let $\Gamma$ be a group and let $F$ be an \ladmissible\ subgroup
  family of~$\Gamma$.  Let~$X$ be a path-connected CW-complex
  with~$\pi_1(X) \cong \Gamma$, and let $U$ be an open $F$-cover
  of~$X$. Then the following hold:
  \begin{enumerate}
  \item If $U$ is convex, then there exists a factorisation
  \[ \xymatrix{%
    H_*(X;\R)
    \ar[r]^-{c^{\ell^1}_X}
    \ar[d]_{H_*(\nu;\R)}
    &
    H^{\ell^1}_*(X;\R)
    \\
    H_*(|N|;\R)
    \ar@{-->}[ur]
    &
    }
  \]
  of the comparison map~$c^{\ell^1}_X$
  through the nerve map~$\nu \colon X \longrightarrow |N|$
  of~$U$.
  \item If the multiplicity of $U$ is at most~$m$, then $c^{\ell^1}_X$
    vanishes in all degrees~$\ast\ge m$.
  \end{enumerate}
\end{thm}
\begin{proof}
  We can argue exactly as in the proof of Theorem~\ref{thm:admissible},
  by working on the chain level instead of the cochain level; instead
  of the mapping theorem in bounded cohomology, we use the corresponding
  mapping theorem in $\ell^1$-homology (Corollary~\ref{cor:mthml1}).
\end{proof}

\subsection{Proof of Theorem~\ref{thm:mainl1}}

As in the case of bounded cohomology, Theorem~\ref{thm:mainl1} now follows
from Theorem~\ref{thm:admissiblel1} and Proposition~\ref{prop:amadmissiblel1}.



\begin{thebibliography}{100}

  \bibitem{bouarich} A.~Bouarich. Th\'eor\`emes de Zilber-Eilemberg et
    de Brown en homologie {$\ell^1$}, \emph{Proyecciones}, 23(2),
    pp.~151--186, 2004.

  \bibitem{dold}
    A.~Dold. \emph{Lectures on Algebraic Topology}, 
    Springer, 1980.
    
  \bibitem{frigeriobc}
    R.~Frigerio.
    \emph{Bounded Cohomology of Discrete Groups},
    Mathematical Surveys and Monographs, 227, AMS, 2017.

  \bibitem{frigeriomoraschini}
    R.~Frigerio, M.~Moraschini.
    Gromov's theory of multicomplexes with applications
    to bounded cohomology and simplicial volume,
    to appear in Mem. Amer. Math. Soc., \textsf{arXiv:1808.07307 [math.GT]}, 2018.

  \bibitem{frigeriol1}
    R.~Frigerio.
    Amenable covers and $\ell^1$-invisibility,
    \emph{J. Topol. Anal.} (in print), \textsf{doi: 10.1142/S1793525320500521}, 2020.
    
  \bibitem{vbc}
    M.~Gromov.
    Volume and bounded cohomology, 
    \emph{Publ.\ Math.\ IHES}, 56, 5--99, 1982.

  \bibitem{ivanov} N.V.~Ivanov. Foundations of the theory of
    bounded cohomology, \emph{J.~Soviet Math.}, 37,
    pp.~1090--1114, 1987.

  \bibitem{ivanovTNG} N.V.~Ivanov.
    Notes on the bounded cohomology theory,
    preprint, \textsf{arXiv:1708.05150 [math.AT]}, 2017.
    
  \bibitem{leary} D.~Juan-Pineda, I.~Leary. 
    On classifying spaces for the family of virtually cyclic subgroups, 
    \emph{Recent developments in algebraic topology,  
    Contemp.~Math.}, 407, 135--145, 2006. 
    
  \bibitem{leray} J.~Leray. Sur la forme des espaces topologiques et sur les points fixes des
    re\-pr\'e\-sentations, \emph{J.~Math.~Pures Appl.}, (9) 24, 95--167, 1945.
    
  \bibitem{loehl1}
    C.~L\"oh.
    Isomorphisms in $\ell^1$-homology,
    \emph{M\"unster J.~Math}, 1, pp.~237--266, 2008.

  \bibitem{loehphd}
    C.~L\"oh.
    \emph{$\ell^1$-Homology and Simplicial Volume},
    PhD~thesis, WWU~M\"unster, 2007.
    \textsf{urn:nbn:de:hbz:6-37549578216}
   
  \bibitem{loehbcd}
    C.~L\"oh.
    A note on bounded-cohomological dimension of discrete groups, 
    \emph{J.~Math.\ Soc.\ Japan}, 69(2), pp.~715--734, 2017.

  \bibitem{lueck}
    W.~L\"uck.
    Transformation Groups and Algebraic K-Theory,
    Volume~1408 of \emph{Lecture Notes in Mathematics}, Springer, 1989. 

    
  \bibitem{luecksurvey}
    W.~L\"uck.
    Survey on classifying spaces for families of subgroups,
    \emph{Infinite groups: geometric, combinatorial and dynamical aspects},
    269--322, 
    \emph{Progress in Mathematics}, 248, Birkh\"auser, 2005. 

  \bibitem{monod} N.~Monod. \emph{Continuous Bounded Cohomology
    of Locally Compact Groups}. Volume~1758 of \emph{Lecture
    Notes in Mathematics}, Springer, 2001.

 \bibitem{tomdieck} T.~tom Dieck. \emph{Transformation groups}. De Gruyter Studies in Mathematics, 8. Walter de Gruyter \& Co., 1987.

  \bibitem{weil} A.~Weil. Sur les th\'eor\`emes de de Rham,
    \emph{Comment. Math. Helv.}, 26, 119--145, 1952. 

\end{thebibliography}
\end{document}